\documentclass[12pt]{amsart}

\usepackage{amsmath,amssymb,setspace,nicefrac, yhmath, amscd,eucal}
\usepackage[active]{srcltx}
\usepackage[colorlinks, linkcolor=red, citecolor=blue, urlcolor=blue, hypertexnames=true]{hyperref}
\usepackage{amsrefs, mathrsfs}

\allowdisplaybreaks

\newtheorem{thm}{Theorem}[section]

\newtheorem{lem}[thm]{Lemma}

\newtheorem{definition}[thm]{Definition}

\newtheorem{proposition}[thm]{Proposition}

\theoremstyle{definition}

\begin{document}
\title[Continuous CSR for coinduced actions and relative ends]{Continuous cocycle superrigidity for coinduced actions and relative ends}
\author{Yongle Jiang}
\date{\today}
\address{Math Department, Sungkyunkwan University, Suwon 440-746, Korea.}
\curraddr{Institute of Mathematics, Polish Academy of Sciences, Warsaw, Poland}
\email{yjiang@impan.pl}

\begin{abstract}
We prove that certain coinduced actions for an inclusion of finitely generated commensurated subgroups with relative one end are continuous cocycle superrigid actions. We also show the necessity for the relative end assumption.
\end{abstract}



\maketitle

\section{Introduction}
This note is one of a series of papers to investigate the connection between continuous orbit equivalence theory (as formulated by Li in \cite{li}) and geometric group theory. Let us first recall the basic definitions briefly.

Let $G$ be a countable (discrete) group which acts on a compact metrizable space $X$ by homeomorphisms. We call a map $c: G\times X\to H$ is a \emph{continuous cocycle} if it is continuous and satisfies the cocycle identity, i.e. $c(st, x)=c(s, tx)c(t, x)$ holds for all $s$, $t$ in $G$ and all $x$ in $X$. Here, $H$ is a group. A continuous cocycle $c$ is called \emph{trivial} if there exist some group homomorphism $\phi: G\to H$ and a continuous map $b: X\to H$ such that $c(s, x)=b(sx)^{-1}\phi(s)b(x)$ for all $s$, $t$ in $G$ and all $x$ in $X$. We are interested in finding \emph{continuous cocycle superrigid actions}, i.e. actions $G\curvearrowright X$ such that every continuous cocycle for it is trivial for all countable group $H$. We refer the readers to \cite{cj, cj1, li, li2} for the background and known results. 
 
Now, let us explain two motivations to study coinduced actions (in the topological setting) in this note.

One comes from our previous work with Chung. In \cite{cj}, we proved that full shifts of non-torsion one-ended groups are continuous cocycle superrigid actions. Later on, Cohen \cite{cohen} removed the non-torsion assumption. Combining it with the work of Li in \cite{li}, this yields new examples of continuous orbit equivalence superrigid actions of a wide class of groups. 

It is well-known that full shifts of amenable groups (or even sofic groups, see \cite{kl}) have positive (topological) entropy, one may wonder whether continuous cocycle superrigid actions with zero entropy exist. In \cite[Example 5.4]{sch}, Schmidt gave an explicit example of this type when the acting group $G$ is $\mathbb{Z}^d$, where $d\geq 2$, but we are unaware of any other examples. Hence, looking for a systematic way to produce such actions, it is natural to consider coinduced actions since 
full shifts can be regarded as degenerate cases of coinduced actions and taking coinduction preserves topological (sofic) entropy (see \cite[Proposition 6.22]{ben}).

The other one is the work of Drimbe. In \cite{d}, he proved a cocycle superrigidity theorem for a large class of coinduced actions in the measurable setting, which generalized the celebrated Popa's cocycle superigidity theorem for Bernoulli shifts. Since our work in \cite{cj} can be thought of as a topological version of Popa's cocycle superrigidity theorem \cite{popa1, popa2}, it is natural to expect a topological version of Drimbe's result should exist.

Motivated by these questions, we generalize the results in \cite{cj, cohen} to coinduced actions using the notion of relative ends, a basic concept in geometric group theory.

\begin{thm}\label{main thm 1}
Let $G$ be a finitely generated infinite group and $K<G$ be a finitely generated commensurated subgroup such that $\widetilde{e}(G, K)=1$. Consider any continuous action $K\curvearrowright X$, where $X$ is any compact metrizable space, if the action has a fixed point, then any continuous cocycle for the coinduced action $G\curvearrowright Y:=X^{G/K}$ into any countable group $H$ is trivial, i.e. it is cohomologous to a group homomorphism from $G$ to $H$ via a continuous transfer map $b: Y\to H$.
\end{thm}
Note that the main results in \cite{cj, cohen} correspond to the degenerate case $K=\{e\}$.

We also generalize \cite[Theorem 2]{cj} to show the necessity on the relative end assumption.
\begin{thm}\label{main thm 2}
Let $K\leq G$ be countable groups. If $\widetilde{e}(G, K)>1$. Then there exists a continuous action $K\curvearrowright X$ with fixed points, where $X$ is a compact metrizable space, such that there exists a non-trivial continuous cocycle for the coinduced action $G\curvearrowright X^{G/K}$ into $\mathbb{Z}_2$. 
\end{thm}

Note that in this paper, $\mathbb{Z}_2$ means $\mathbb{Z}/2\mathbb{Z}$. The outline of this paper is as follows. 

In Section \ref{section on preliminaries}, we give most of our definitions and some basic facts. In Section \ref{section on geometry of coset graph}, we review the geometry of the coset graph, which was studied in \cite{cm}. Following the proof in \cite{cohen}, we prepare some lemmas in Section \ref{section on cocycles over the coinduced actions} and then prove Theorem \ref{main thm 1} in Section \ref{section on the proof of theorem 1}. In Section \ref{section on proof of theorem 2}, we prove Theorem \ref{main thm 2}. Then we conclude with some remarks in Section \ref{section on remarks}. 

\section{Preliminaries}\label{section on preliminaries}
In this section, we recall the basic concepts that are used later.

\subsection*{Commensurated subgroups}
If $K$ is a subgroup of $G$, then the commensurator of $K$ in $G$ is 
$Comm_G({K}):=\{g\in G~|~ gK g^{-1}\cap K ~\mbox{has finite index in both $K$ and $gK g^{-1}$}\}.$
A subgroup $K$ is \emph{commensurated} in $G$ if $Comm_G({K})=G$. The term ``commensurated" was used in \cite{sw} and we follow that terminology rather than saying $(K, G)$ is a ``Hecke" pair as in \cite{kri, tz} or $K$ is almost normal in $G$ as in \cite{anan}.

One main property of commensurated subgroups we use is the following, see \cite[Section 2]{cm}.

\begin{proposition}\label{property of commensurators}
Let $K$ be a subgroup of $G$. Then $Comm_G({K})=G$ if and only if for any $g\in G$, there is a finite subset $F_g\subseteq G$ such that $Kg\subseteq F_gK:=\cup_{s\in F_g}sK$. 
\end{proposition}
\begin{proof}
The ``only if" direction is easy. For the ``if" direction, see \cite[Theorem 2.3]{cm}. For completeness, we sketch the proof here. Fix any $g\in G$, the assumption is $g^{-1}K\subseteq KF_g^{-1}$. Then define a map $t: K\to F_g$ such that $g^{-1}k\in Kt(k)^{-1}$. Once can check that $t(k)=t(k')$ implies $(K\cap gKg^{-1})k=(K\cap gKg^{-1})k'$. Write $t(K)=\{t(k_1), \dots, t(k_n)\}$, then $K=\cup_{i}(K\cap gKg^{-1})k_i$. Hence $[K: K\cap gKg^{-1}]<\infty$; similarly, $[K: K\cap g^{-1}Kg]=[gKg^{-1}: K\cap gKg^{-1}]<\infty$.
\end{proof}

We list some classical examples of commensurated subgroups below.

\begin{itemize}
\item $K$ be any finite subgroup or subgroup with finite index.
\item $K$ be any normal subgroup.
\item $K=GL_n(\mathbb{Z})<G=GL_n(\mathbb{Q})$ for all $n$ by \cite{kri}.
\item $K=\mathbb{Z}\rtimes \{1\}$, $G=\mathbb{Q}\rtimes \mathbb{Q}_+^*$ \cite{bc}.
\item $K=\langle x\rangle<G=BS(m, n)=\langle t, x:~ t^{-1}x^mt=x^n \rangle$ by \cite[Example 3.1]{cm}. 
\item More generally, let $G$ act by isometries on a locally finite metric space (here, a metric space is \emph{locally finite} if its balls have a finite number of elements), then take $K$ to be the stabilizer of any vertex (\cite[Remark 1.3]{anan}).
\end{itemize}
Using these examples, one can build more examples using basic properties of commensurators, see e.g. \cite{cm, tz}.

\subsection*{Coinduced actions}
Let $K<G$ be a subgroup. Fix any lift $L: G/K\to G$, i.e. $L$ is a map such that $L(sK)K=sK$ for all $s\in G$. Then since $gL(sK)K=gsK=L(gsK)K$, we get $L(gsK)^{-1}gL(sK)\in K$, which is denoted by $\delta(g, sK)$. One can check that $\delta: G\times G/K\to K$ is a cocycle, i.e. $\delta(g_1g_2, sK)=\delta(g_1, g_2sK)\delta(g_2, sK)$ holds for all $g_1$, $g_2$ and $s$ in $G$.

Let $\alpha: K\curvearrowright X$ be a continuous action, where $X$ is a compact metrizable space, the \emph{coinduced action} $G\curvearrowright Y:=X^{G/K}$ is defined as $(gy)_{sK}:=\alpha(\delta(g, g^{-1}sK))y_{g^{-1}sK}$ for all $s, g\in G$ and $y\in Y$, where $Y$ is given the product topology. It is routine to check that the coinduced action is isomorphic to the left shift action of $G$ on the subspace $Z:=\{z\in X^G: z_{gk}=\alpha(k^{-1})z_g, ~\mbox{for all $g\in G$, $k\in K$}\}$ of $X^G$ and the choice of lift $L$ does not affect the coinduced action (up to conjugacy).

\begin{proposition}\label{fixed pts for coinduced actions}
The coinduced action $G\curvearrowright X^{G/K}$ has a fixed point if and only if there exists some $x\in X$ such that for all $sK=tK$, $\alpha(\delta(s, K))x=\alpha(\delta(t, K))x$. In particular, if $K\curvearrowright X$ has a fixed point, then the coinduced action has a fixed point.
\end{proposition}
\begin{proof}
For the ``if" direction, check that $y_{sK}:=\alpha(\delta(s, K))x$ defines a fixed point for the coinduced action. For the ``only if" direction, let $y$ be a fixed point of the coinduced action, then for all $g$, $s$ in $G$, $y_{sK}=\alpha(\delta(g, g^{-1}sK))y_{g^{-1}sK}$. In particular, take $g=s$, we deduce that $y_{sK}=\alpha(\delta(s, K))y_{K}$; hence, for all $sK=tK$, $\alpha(\delta(s, K))y_{K}=\alpha(\delta(t, K))y_{K}$. Take $x=y_{K}$.
\end{proof}

\begin{proposition}\label{finite union implies finite index}
Let $K\leq G$ be groups. If there exists some finite subset $F$ of $G$ such that $G=\cup_{s, t\in F}sKt$. Then $[G:K]<\infty$. 
\end{proposition}
\begin{proof}
Write $K_t:=t^{-1}Kt$ for $t\in F$, then $G=\cup_{t\in F}F_tK_t$, where $F_t:=Ft$. By \cite[Lemma 2.7]{d}, we deduce that for some $t$, $[G: K_t]<\infty$, hence $[G: K]<\infty$.
\end{proof}

\begin{proposition}\label{product measure for coinduced actions}
Let $\mu$ be a $K$-invariant probability measure on $X$. Then 

(1) the product measure $\nu:=\mu^{G/K}$ is invariant under the coinduced action. 

(2) the coinduced action is weakly mixing with respect to $\nu$ if $K$ is a subgroup of $G$ with infinite index. 
\end{proposition}

\begin{proof}
To prove (1), take any standard rectangle $A=\prod_{g\in F}A_{gK}$, where $A_{gK}\subseteq X$ is a measurable set. One can check that $gA=\prod_{gs\in gF}\alpha(\delta(g, sK))A_{sK}$. Hence, $\nu(A)=\nu(gA)=\prod_{g\in F}\mu(A_{gK})$. Then apply \cite[Proposition 2.30]{kl} to finish the proof.

To prove (2), take any finitely many standard rectangles $A_i=\prod_{g\in F_i}A_{gK}$ and $B_i=\prod_{g\in F'_i}B_{gK}$, it suffices to show that for any $\epsilon>0$, there exists some $g\in G$, such that $|\nu(A_i\cap gB_j)-\nu(A_i)\nu(B_j)|<\epsilon$ for all $i, j$. Clearly, one just need to show that there exists some $g\in G$ such that $F_i K\cap g F'_j K=\emptyset$, i.e. $g\not\in \cup_{i, j} F_i K F^{'-1}_j$, which is possible by Proposition \ref{finite union implies finite index}.
\end{proof}

\subsection*{Relative ends}
For a finitely generated group $G$, a classical notion of ends was developed by Hopf \cite{hopf} and Freudenthal \cite{freu}. Given a subgroup $K$ of $G$, one may ask whether a relative notion of ends could be developed. Indeed, there are already three different definitions of relative ends for a pair of groups $K<G$. One is due to Houghton \cite{hou} and Scott \cite{sco}, which is denoted by $e(G, K)$; another one is due to Kropholler and Roller \cite{kr}, which is denoted by $\widetilde{e}(G, K)$ and the last one is due to Andrade and Fanti \cite{af}. In this note, it suffices to introduce the first two definitions for use. For more discussion of these notions, see \cite[Section 14.5]{gtm243}. 

Let $S$ be any set, $G$ be a group and $K$ be a subgroup. Let $\mathfrak{P}S$ and $\mathfrak{F}S$ denote the power set of $S$ and the set of finite subsets of $S$ respectively. Let $\mathfrak{F}_{K}G$ denote the set of $K$-finite subset of $G$, i.e. $\mathfrak{F}_{K}G:=\{A\subseteq G~|~A \subseteq FK~\mbox{for some finite subset $F$ of $G$}\}$.  
Both $\mathfrak{P}G$, $\mathfrak{F}G$ and $\mathfrak{F}_{K}G$ admit the action of $G$ by left multiplication, and can be regarded as left $G$-modules over the field of two elements, which is denoted by $\mathbb{Z}_2$. Then, one defines the algebraic end invariant as 
$$\widetilde{e}(G, K):=\dim_{\mathbb{Z}_2}(\mathfrak{P}G/\mathfrak{F}_{K}G)^G.$$
Note that Geoghegan gave a topological description of $\widetilde{e}(G, K)$, which is called the number of filtered ends of $(G, K)$ in  \cite[Section 14.5]{gtm243} and \cite[Section 5]{cm}.

The geometric end invariant is defined as 
$e(G, K):=\dim_{\mathbb{Z}_2}(\mathfrak{P}(G/K )/\mathfrak{F}(G/K))^G.$
Here, $G/K$ denote the quotient of $G$ by the right action of $K$.

We collect together the properties of end invariants defined above that we may need later.

\begin{proposition}\label{property of relative ends}
Let $K\leq G$ be groups. Then the following hold.\\
(1) $e(G, K)=0$ if and only if $[G: K]<\infty$ if and only if $\widetilde{e}(G, K)=0$.\\
(2) If $K$ has infinite index in $G$, then $\widetilde{e}(G, K)=1+dim_{\mathbb{Z}_2}H^1(G, \mathfrak{F}_{K}G)$.\\
(3) $e(G, K)\leq \widetilde{e}(G, K)$ and $\widetilde{e}(G, K)=1$ implies $e(G, K)=1$.\\
(4) $e(G, K)=e(X)$, where $X$ is the left coset graph of $G$ with respect to $K$.\\
(5) Let $G$ and $K$ be finitely generated and suppose that $K$ has infinite index in $Comm_G(K)$. Then $\widetilde{e}(G, K)$ is either 1, 2 or infinite.\\
(6) If $K$ has infinite index in its normalizer in $G$. Then $e(G, K)$ is either 1, 2 or infinite.\\
(7) Let $K'\leq G$. If $K$ and $K'$ are commensurable, i.e. both $[K: K\cap K']<\infty$ and $[K': K\cap K']<\infty$, then $\widetilde{e}(G, K)=\widetilde{e}(G, K')$. In particular, if $K$ is finite, then $\widetilde{e}(G, K)=\widetilde{e}(G, 1)=e(G)$.\\
(8) If $K$ is finitely generated and normal in $G$. Then $\widetilde{e}(G, K)=e(G/K)$.
\end{proposition}

\begin{proof}
(1) is clear from the definition. (2) is Lemma 1.2 in \cite{kr}. (3) is Lemma 2.5 (i) in \cite{kr} and by (1). (4) is Lemma (ii) in \cite{sco}. (5) is Theorem 1.3 in \cite{kr}. (6) is Theorem 13.5.21 in \cite{gtm243}. (7) is clear since if $K$ and $K'$ are commensurable, then $\mathcal{F}_K(G)=\mathcal{F}_{K'}(G)$. (8) is Lemma 2.4 (iv) in \cite{kr}.
\end{proof}

\section{Geometry of the coset graph}\label{section on geometry of coset graph}
For the proof of Theorem \ref{main thm 1}, we would follow the approach in \cite{cohen} (one can also follow \cite{cj} but under further assumptions on $G/K$). However, we would work with the coset graph of $G$ with respect to $K$, which was defined on page 124 of \cite{cm}, rather than the Cayley graph of $G$ as in \cite{cj, cohen} when dealing with full shifts. Let us recall the definition of this coset graph below.

Suppose $G$ is a group with finite generating set $S$ and $K$ is a subgroup of $G$. Let $\Lambda(S, K, G)$ be the (left) coset graph with vertices the left cosets $gK$ of $G$ and a directed edge (labeled $s$) from $gK$ to $fK$ ($gK\neq fK$) if for some $s\in S$, $gKs\cap fK\neq \emptyset$. One can define a metric on $\Lambda(S, K, G)$ using the edge path metric, i.e. $d(g_1K, g_2K)$ is defined as the minimal length of paths connecting $g_1K$ and $g_2K$. Note that $G$ acts (on the left) transitively on the vertices of $\Lambda(S, K, G)$  by isometries.

We also define the word norm $|gK|:=d(gK, K)$, where the neutral element of $\Lambda(S, K, G)$, $1_{\Lambda(S, K, G)}$, is just the coset $K$.

For any $T\subseteq \Lambda(S, K, G)$, let $\mathcal{N}_LT:=\{gK :~\inf\{d(gK, s):~s\in T\}\leq L\}$ denote the $L$-neighborhood of $T$. Let $B(r, gK)$ denote the closed $r$-ball around $gK\in \Lambda(S, K, G)$, i.e. $B(r, gK):=\mathcal{N}_{r}\{gK\}$.

The crucial properties of this coset graph we use are the following.

\begin{proposition}\label{property of 2nd coset graph}
(1) Let $G$ be a group with finite generating set $S$. If $K$ is commensurated in $G$, then $\Lambda(S, K, G)$ is locally finite and connected.\\
(2) Let $G$ be a group with finite generating set $S$, $K$ be a finitely generated commensurated subgroup of $G$ with generating set a subset of $S$. Then the number of ends of $\Lambda(S, K, G)$ is equal to $\widetilde{e}(G, K)$.
\end{proposition}
\begin{proof}
(1) $\Lambda(S, K, G)$ is locally finite by Theorem 4.4 in \cite{cm}. From Corollary 2.4 in \cite{cm}, we know for all $g\in G$, there exist finitely many elements in $S$, say $s_1, \dots, s_n$, such that $Ks_1\cdots s_n\cap gK\neq \emptyset$. Take $k_1, k_2\in K$, such that $gk_2=k_1s_1\cdots s_n$. Then $K, k_1s_1K, k_1s_1s_2K, \dots, k_1s_1\cdots s_{n-1}K, gK$ is a path connecting $K$ and $gK$ via edges $s_1, s_2, \dots, s_n$. Hence, $\Lambda(S, K, G)$ is connected.

(2) is Theorem 5.4 in \cite{cm}.
\end{proof}

For the rest of the paper, we would assume $K\leq G$ is commensurated, both $G$ and $K$ are finitely generated and $\widetilde{e}(G, K)=1$. Note that $[G: K]=\infty$ now.

By Proposition \ref{property of 2nd coset graph} (2), we know that for any $r$, $\Lambda(S, K, G)$ has exactly one 
unbounded connected component after removing a finite ball.

\begin{definition}\label{def 2.6}
For $r\in\mathbb{N}$, let\\
$N(r):=\sup\{|gK|: gK ~\mbox{does not belong to the unbounded component of}~ \Lambda(S, K, G)\setminus B(r, 1_{\Lambda(S, K, G)})\}$.
\end{definition}
Since $\Lambda(S, K, G)$ is locally finite, we have $N(r)<\infty$. 

\begin{lem}\label{lemma 2.2}
If $g\in G$, there exists a geodesic $\gamma: \{0, \dots, |gK |\}\to \Lambda(S, K, G)$ such that $\gamma(0)=1_{\Lambda(S, K, G)}$ and $\gamma(|gK|)=gK$.
\end{lem}
\begin{proof}
Similar to the proof of Lemma 2.2 in \cite{cohen}.
\end{proof}

\begin{lem}\label{lemma 2.3}
There exists a geodesic $\gamma: \mathbb{Z}\to \Lambda(S, K, G)$ with $\gamma(0)=1_{\Lambda(S, K, G)}$.
\end{lem}
\begin{proof}
Recall that $\widetilde{e}(G, K)\neq 0$ implies $[G: K]=\infty$ by Proposition \ref{property of relative ends} (1), and since $\Lambda(S, K, G)$ is locally finite by Proposition \ref{property of 2nd coset graph} (1), the word norm on $\Lambda(S, K, G)$ achieves every natural number. For each $r\in \mathbb{N}$, choose some $g_{2r} K\in \Lambda(S, K, G)$ such that $|g_{2r}K|=2r$. By Lemma \ref{lemma 2.2}, there is a geodesic $\gamma_r: \{0, \dots, 2r\}\to \Lambda(S, K, G)$ such that $\gamma_r(0)=1_{\Lambda(S, K, G)}$ and $\gamma_r(2r)=g_{2r}K$. Write $\gamma_r(i)=g_i K$. Let $\widetilde{\gamma}_r: \{-r, \dots, r\}\to \Lambda(S, K, G)$ be defined by $\widetilde{\gamma}_r(n):=g_r^{-1}g_{n+r}K$, so that $\widetilde{\gamma}_r$ is a geodesic and $\widetilde{\gamma}_r(0)=1_{\Lambda(S, K, G)}$. By Arzela-Ascoli theorem, the sequence $(\widetilde{\gamma}_r)_{r\in \mathbb{N}}$ subconverges pointwise to a geodesic $\gamma: \mathbb{Z}\to \Lambda(S, K, G)$. 
\end{proof}

Let $\mathbb{Z}_{\geq 0}=\{n\in \mathbb{Z}: n\geq 0\}$ and  $\mathbb{Z}_{\leq 0}=\{n\in \mathbb{Z}: n\leq 0\}$. Recall that $\mathcal{N}_{L}(T)$ denotes the $L$-neighborhood of $T\subseteq \Lambda(S, K, G)$.

\begin{lem}\label{lemma 2.4}
If $\gamma: \mathbb{Z}\to \Lambda(S, K, G)$ is a geodesic with $\gamma(0)=1_{\Lambda(S, K, G)}$, then $\mathcal{N}_L(\gamma(\mathbb{Z}_{\geq 0}))\cap \mathcal{N}_L(\gamma(\mathbb{Z}_{\leq 0}))\subseteq B(3L, 1_{\Lambda(S, K, G)})$.
\end{lem}
\begin{proof}
Similar to the proof of Lemma 2.4 in \cite{cohen}.
\end{proof}

\section{Cocycles over the coinduced actions}\label{section on cocycles over the coinduced actions}
Fix a continuous cocycle for the coinduced action $c: G\times X^{G/K}\to H$, where $H$ is a countable group. Here, $c$ is continuous means that for every $g\in G$, there is a finite subset $F\subseteq G$ (depending on $g$) such that $c(g, y)$ depends only on the restriction $y|_{FK}$ of $y$ to the finitely many cosets $FK$. That is, if $y|_{FK}=z|_{FK}$, then $c(g, y)=c(g, z)$.

\begin{definition}\label{def 3.1}
Let $S$ and $T$ be the symmetric generating set of $G$ and $K$ respectively, and we may assume $T\subseteq S$. Let $L\in\mathbb{N}$ be such that for any generator $s\in S$, the function $c(s, y)$ is determined by $y|_{B(L, 1_{\Lambda(S, K, G)})}$.
\end{definition}

\begin{definition}\label{def 3.2}
Fix a $K$-invariant point $x_0\in X$, and denote by $\overline{0}\in X^{G/K}$ to be the $G$-invariant point, i.e.  $\overline{0}_{gK}=x_0$ for all $g\in G$. Let $\Delta(\overline{0})\subseteq X^{G/K}$ consist of all $y\in X^{G/K}$ such that $y_{gK}=x_0$ for all but finitely many $g\in G$, and for $y\in \Delta(\overline{0})$, let $||y||$ denote $\sup\{|gK|:~ y_{gK}\neq x_0\}$.
\end{definition}

\begin{lem}\label{lem 3.3}
Let $\gamma:\{0,\ldots,n\}\to \Lambda(S, K, G)$ be a path, say, $\gamma(i)=g_iK$ for $i=0, \dots, n$, and take $y\in X^{G/K}$. Then $c(g_n^{-1},y)c(g_0^{-1},y)^{-1}$ is determined by $y|_{\mathcal{N}_L(\gamma\{0,\ldots,n\})}.$ That is, if 
$$y|_{\mathcal{N}_L(\gamma\{0,\ldots,n\})}=z|_{\mathcal{N}_L(\gamma\{0,\ldots,n\})},$$
then
$$c(g_n^{-1},y)c(g_0^{-1},y)^{-1}=c(g_n^{-1},z)c(g_0^{-1},z)^{-1}.$$
\end{lem}

\begin{proof}
By the definition of a path in $\Lambda(S, K, G)$, there exist $s_i, i=0, \dots, n$ in $S$ such that $g_{i+1}K\cap g_iKs_i\neq \emptyset$ for all $i=0, \dots, n-1$. Hence, there are $x_i, y_i$ in $K$ such that $g_{i+1}x_i^{-1}=g_iy_is_i$ for all $i=0, \dots, n-1$. Hence, $g_i=g_0\prod_{j=0}^{i-1}(y_js_jx_j)$ for all $i=1,\dots, n$. In particular, $g_n^{-1}=(\prod_{i=n-1}^0x_i^{-1}s_i^{-1}y_i^{-1})g_0^{-1}$,

\begin{eqnarray*}
&&c(g_n^{-1}, y)c(g_0^{-1}, y)^{-1}=\prod_{i=n-1}^0\eta_i^{(1)}(y)\eta_i^{(2)}(y)\eta_i^{(3)}(y).
\end{eqnarray*}
Here 
\begin{eqnarray*}
\eta_i^{(1)}(y)&:=&c(x_i^{-1}, s_i^{-1}y_i^{-1}\prod_{j=i-1}^0(x_j^{-1}s_j^{-1}y_j^{-1})g_0^{-1}y),\\
\eta_i^{(2)}(y)&:=&c(s_i^{-1}, y_i^{-1}\prod_{j=i-1}^0(x_j^{-1}s_j^{-1}y_j^{-1})g_0^{-1}y),\\
\eta_i^{(3)}(y)&:=&c(y_i^{-1}, \prod_{j=i-1}^0(x_j^{-1}s_j^{-1}y_j^{-1})g_0^{-1}y).
\end{eqnarray*}
Similarly, we can define $\eta_i^{(k)}(z)$. We claim that for each $i$ and $k$, $\eta_i^{(k)}(y)=\eta_i^{(k)}(z)$. This would finish the proof.

Fix any $i\in \{0, \dots, n-1\}$.

For $k=1$. Since $K$ is finitely generated, we can write $x_i^{-1}=t_1\cdots t_l$ for some $t_*$ in $T$. Then, 
\begin{eqnarray*}
\eta_i^{(1)}(y)=\prod_{m=1}^lc(t_m, t_{m+1}\cdots t_ls_i^{-1}y_i^{-1}\prod_{j=i-1}^0(x_j^{-1}s_j^{-1}y_j^{-1})g_0^{-1}y).
\end{eqnarray*}
If suffices to check that for each $m$,
\begin{eqnarray*}
&&c(t_m, t_{m+1}\cdots t_ls_i^{-1}y_i^{-1}\prod_{j=i-1}^0(x_j^{-1}s_j^{-1}y_j^{-1})g_0^{-1}y)\\
&=&c(t_m, t_{m+1}\cdots t_ls_i^{-1}y_i^{-1}\prod_{j=i-1}^0(x_j^{-1}s_j^{-1}y_j^{-1})g_0^{-1}z).
\end{eqnarray*}
For this, we just need to check 
$(t_{m+1}\cdots t_ls_i^{-1}y_i^{-1}\prod_{j=i-1}^0(x_j^{-1}s_j^{-1}y_j^{-1})g_0^{-1}y)|_{B(L, 1_{\Lambda(S, K, G)})}=(t_{m+1}\cdots t_ls_i^{-1}y_i^{-1}\prod_{j=i-1}^0(x_j^{-1}s_j^{-1}y_j^{-1})g_0^{-1}z)|_{B(L, 1_{\Lambda(S, K, G)})}$; 

Then, it suffices to check  

$(t_{m+1}\cdots t_ls_i^{-1}y_i^{-1}\prod_{j=i-1}^0(x_j^{-1}s_j^{-1}y_j^{-1})g_0^{-1})^{-1}B(L, 1_{\Lambda(S, K, G)})\subseteq \mathcal{N}_L(\gamma\{0,\ldots,n\})$.

To see this, take any $gK\in B(L, 1_{\Lambda(S, K, G)})$,
we have the following:

\begin{eqnarray*}
&&d((t_{m+1}\cdots t_ls_i^{-1}y_i^{-1}\prod_{j=i-1}^0(x_j^{-1}s_j^{-1}y_j^{-1})g_0^{-1})^{-1}gK, \gamma(i+1))\\
&=&d(g_0\prod_{j=0}^{i-1}(y_js_jx_j)y_is_it_l^{-1}\cdots t_{m+1}^{-1}gK, g_{i+1}K)\\
&=&d(g_iy_is_it_l^{-1}\cdots t_{m+1}^{-1}gK, g_{i+1}K)\\
&=&d(g_{i+1}x_i^{-1}t_l^{-1}\cdots t_{m+1}^{-1}gK, g_{i+1}K)\\
&=&d(x_i^{-1}t_l^{-1}\cdots t_{m+1}^{-1}gK, K)~\mbox{(since $d$ is left $G$-invariant)}\\
&=&d(gK, t_{m+1}\cdots t_jx_iK)\\
&=&d(gK, K)\leq L. ~\mbox{(since $t_*, x_i\in K$)}
\end{eqnarray*}
One can prove the case $k=3$ similarly.

For $k=2$, it suffices to check $(y_i^{-1}\prod_{j=i-1}^0(x_j^{-1}s_j^{-1}y_j^{-1})g_0^{-1})^{-1}B(L, 1_{\Lambda(S, K, G)})\subseteq \mathcal{N}_L(\gamma\{0,\ldots,n\})$.
The proof is similar to the above by using the fact that $y_i\in K$.
\end{proof}


\section{Proof of Theorem \ref{main thm 1}}\label{section on the proof of theorem 1}
We follow the proof in \cite[Section 4]{cohen} and do some mild modification.

Recall the notations $\overline{0}$, $\Delta(\overline{0})$ and $||y||$ from Definition \ref{def 3.2} and the notation $N(r)$ from Definition \ref{def 2.6}.

\begin{definition}\label{def 4.1}
Let $\phi(g):=c(g,\overline{0})$. For each $y\in\Delta(\overline{0})$, fix some $g_y\in G$ such that $\min\{|g_yK|,|g_y^{-1}K|\}>N(\|y\|+L)$, then set $b(y):=c(g_y,y)^{-1}\phi(g_y)$.
\end{definition}

Note that such a $g_y$ exists by Proposition \ref{property of commensurators} or Proposition \ref{finite union implies finite index}.

Observe that $\phi$ is a homomorphism by the cocycle identity, since $\overline{0}$ is a fixed point for the coinduced action. We will show that the restriction of $c$ to $G\times \Delta(\overline{0})$ is cohomologous to $\phi$ with transfer function $b$, and then we will show that $b$ extends continuously to all of $Y=X^{G/\Lambda}$. First, we see that the choice of $g_y$ in Definition \ref{def 4.1} is irrelevant.

\begin{lem}\label{lem 4.2}
If $\min\{|gK|, |g^{-1}K|\}>N(\|y\|+L)$, then $c(g,y)^{-1}\phi(g)=b(y)$.
\end{lem}

\begin{proof}
Let $\gamma:\{0,\ldots,n\}\to \Lambda(S, K, G)$ be a path such that $\gamma$ connects $g^{-1}K$ to $g_y^{-1}K$ outside of $B(\|y\|+L,1_{\Lambda(S, K, G)})$ (so that $|\gamma(j)|>\|y\|+L$ for $j=0,\ldots, n$). Such a path exists because $|g^{-1}K|$ and $|g_y^{-1}K|$ are greater than $N(\|y\|+L)$. Since $|\gamma(j)|>\|y\|+L$ for all $j\in \{0,\ldots,n\}$ and $y_{gK}=x_0$ whenever $|gK|>\|y\|$, we have that $y_{gK}=x_0$ for any $gK\in\mathcal{N}_L(\gamma\{0,\ldots,n\})$, so by Lemma \ref{lem 3.3}, we have

$$c(g_y,y)c(g,y)^{-1}
=c(g_y,\overline{0})c(g,\overline{0})^{-1}
.$$
Consequently
$$c(g,y)^{-1}\phi(g)=c(g_y,y)^{-1}\phi(g_y)=b(y).\qedhere$$
\end{proof}

We now show that the restriction of $c$ to $G\times\Delta(\overline{0})$ is cohomologous to $\phi$ with transfer function $b$.

\begin{proposition}
\label{prop 4.3}
For all $g\in G$ and $y\in\Delta(\overline{0})$, we have $c(g,y)=b(gy)\phi(g)b(y)^{-1}$.
\end{proposition}

\begin{proof}
Choose some $\tilde{g}\in G$ such that $\min\{|\tilde{g}g^{-1}K|, |g\tilde{g}^{-1}K|\}>N(\|gy\|+L)$ and $\min\{|\tilde{g}K|, |\tilde{g}^{-1}K|\}>N(\|y\|+L)$. This $\tilde{g}$ exists by Proposition \ref{property of commensurators} or Proposition \ref{finite union implies finite index}.
So by Lemma \ref{lem 4.2} we have $b(gy)=c(\tilde{g} g^{-1},gy)^{-1}\phi(\tilde{g} g^{-1})$ and $b(y)=c(\tilde{g},y)^{-1}\phi(\tilde{g})$.
It follows that $c(\tilde{g} g^{-1},gy)^{-1}=b(gy)\phi(\tilde{g} g^{-1})^{-1}$ and $c(\tilde{g},y)=\phi(\tilde{g})b(y)^{-1}$.

By the cocycle identity,
$$c(\tilde{g}, y)=c(\tilde{g} g^{-1},gy)c(g,y).$$
Rearranging, we obtain
$$c(g, y)=c(\tilde{g} g^{-1},gy)^{-1}c(\tilde{g},y)
=b(gy)\phi(g)b(y)^{-1}.\qedhere$$
\end{proof}

We now show that $y\mapsto b(y)$ factors through the restriction $y\mapsto y|_{B(3L, 1_{\Lambda(S, K, G)})}$, so that $b$ may be extended continuously to a function $Y\to H$.
\begin{proposition}\label{prop 4.4}
If $y',y\in\Delta(\overline{0})$ with $y'|_{B(3L, 1_{\Lambda(S, K, G)})}=y|_{B(3L, 1_{\Lambda(S, K, G)})}$, then $b(y')=b(y)$.
\end{proposition}
\begin{proof}
By Lemma \ref{lemma 2.3} there exists a geodesic $\gamma:\mathbb{Z}\to \Lambda(S, K, G)$. We will begin by finding 
$z\in\Delta(\overline{0})$ such that $z|_{\mathcal{N}_L(\gamma(\mathbb{Z}_{\geq 0}))}=y'|_{\mathcal{N}_L(\gamma(\mathbb{Z}_{\geq 0}))}$ and $z|_{\mathcal{N}_L(\gamma(\mathbb{Z}_{\leq 0}))}=y|_{\mathcal{N}_L(\gamma(\mathbb{Z}_{\leq 0}))}$.

Let $z_{gK}:=x_0$ for $gK\notin \mathcal{N}_L(\gamma(\mathbb{Z}_{\geq 0}))\cup\mathcal{N}_L(\gamma(\mathbb{Z}_{\leq 0}))$, and take $z_{gK}:=y'_{gK}$ for $gK\in\mathcal{N}_L(\gamma(\mathbb{Z}_{\geq 0}))$ and $z_{gK}:=y_{gK}$ for $gK\in\mathcal{N}_L(\gamma(\mathbb{Z}_{\leq 0}))$. For $gK\in\mathcal{N}_L(\gamma(\mathbb{Z}_{\geq 0}))\cap\mathcal{N}_L(\gamma(\mathbb{Z}_{\leq 0}))\subseteq B(3L, 1_{\Lambda(S, K, G)})$ (recall Lemma \ref{lemma 2.4}), these definitions agree by our assumption that $y'$ and $y$ agree on $B(3L, 1_{\Lambda(S, K, G)})$. Because $y',y\in\Delta(\overline{0})$, we also have $z\in\Delta(\overline{0})$.

By Lemma \ref{lem 3.3}, $c(g_n^{-1},y')=c(g_n^{-1},z)$ for all $n\geq 0$, where $\gamma(n)=g_nK$. By Lemma \ref{lem 4.2}, this implies that $b(y')=b(z)$. Arguing similarly with $\gamma(-n)$ in place of $\gamma(n)$, we see that $b(y)=b(z)$. It follows that $b(y')=b(y)$.
\end{proof}

Now, we are ready to prove Theorem \ref{main thm 1}.

\begin{proof}[Proof of Theorem \ref{main thm 1}]
Take $b$ and $\phi$ as given in Definition \ref{def 4.1}. We will first note that $b$ extends continuously to all of $Y$, and then we will show that this extension is a transfer function for a cohomology from $c$ to $\phi$. 

\textbf{Extending $b$ over $Y$.}
For $y\in Y$, let $\overline{y}\in\Delta(\overline{0})$ be given by setting $\overline{y}_{gK}:=y_{gK}$ for $gK\in B(3L, 1_{\Lambda(S, K, G)})$ and $\overline{y}_{gK}:=x_0$ for $gK\not\in B(3L, 1_{\Lambda(S, K, G)})$.  For arbitrary $y\in Y$ we may now define $b(y):=b(\overline{y})$. By definition $b$ is continuous, and by Proposition \ref{prop 4.4}, we have that $b:Y\to H$ agrees with Definition \ref{def 4.1} on $\Delta(\overline{0})$. We now proceed to show that $c(g,y)=b(gy)\phi(g)b(y)^{-1}$ for every $g\in G$ and $y\in Y$.

\textbf{Showing that $c$ is cohomologous to $\phi$ with transfer function $b$.} Given $g\in G$ and $y\in Y$, we shall apply Proposition \ref{prop 4.3} to some $\tilde{y}\in\Delta(\overline{0})$ approximating $y$. 

Let $\tilde{y}\in\Delta(\overline{0})$ be given by setting $\tilde{y}_{g'K}:=y_{g'K}$ for $g'K\in B(|g^{-1}K|+3L,1_{\Lambda(S, K, G)})$ and $\tilde{y}_{g'K}:=x_0$ for $g'K\not\in B(|g^{-1}K|+3L,1_{\Lambda(S, K, G)})$. By Lemma \ref{lemma 2.2}, there is a geodesic $\gamma:\{0,\ldots,|g^{-1}K|\}\to \Lambda(S, K, G)$ such that $\gamma(0)=1_{\Lambda(S, K, G)}=K$ and $\gamma(|g^{-1}K|)=g^{-1}K$. 

For such a geodesic, $\mathcal{N}_L(\gamma\{0,\ldots,|g^{-1}K|\})\subseteq B(|g^{-1}K|+3L,1_{\Lambda(S, K, G)})$. It follows by Lemma \ref{lem 3.3} that $c(g,y)c(1_G,y)^{-1}=c(g,\tilde{y})c(1_G,\tilde{y})^{-1},$ or in other words,
 $$c(g,y)=c(g,\tilde{y}).$$ 

We have $c(g,\tilde{y})=b(g\tilde{y})\phi(g)b(\tilde{y})^{-1}$ by Proposition \ref{prop 4.3}, as $\tilde{y}\in\Delta(\overline{0})$.

We have $b(\tilde{y})=b(\overline{y})=b(y)$ by definition. We see that
$$(g\tilde{y})|_{B(3L,1_{\Lambda(S, K, G)})}=(gy)|_{B(3L,1_{\Lambda(S, K, G)})}$$
because $g^{-1}B(3L,1_{\Lambda(S, K, G)})\subseteq B(|g^{-1}K|+3L,1_{\Lambda(S, K, G)})$. Thus, $b(g\tilde{y})=b(gy)$. It follows that
$$c(g,y)=c(g,\tilde{y})=b(g\tilde{y})\phi(g)b(\tilde{y})^{-1}=b(gy)\phi(g)b(y)^{-1}.\qedhere$$\qedhere
\end{proof}


\section{Proof of Theorem \ref{main thm 2}}\label{section on proof of theorem 2}

The proof is a direct generalization of the proof of Theorem 2 in \cite{cj}.

\begin{proof}[Proof of Theorem \ref{main thm 2}]
From Proposition \ref{property of relative ends} (2), we know that $H^1(G, \mathcal{F}_{K}G)$ is non-trivial if $\widetilde{e}(G, K)>1$. This means there is a cocycle $c: G\to \mathcal{F}_{K}G$ which is not a coboundary. 

Note that $\mathcal{F}_{K}G\cong \mathbb{Z}G\otimes_{\mathbb{Z}K}\mathbb{Z}_2^{K}$ as $\mathbb{Z}G$-modules by sending $A\in \mathcal{F}_{K}G$ to $\sum_is_i\otimes 1_{K_i}$, where $A=\sqcup_is_iK_i$ for finitely many $s_i\in G$, $K_i\subseteq K$.
Since $G$ is countable, we may find a countable $\mathbb{Z}K$-submodule $P$ of $\mathbb{Z}_2^{K}$ such that $c(G)\subseteq \mathbb{Z}G\otimes_{\mathbb{Z}K}P\subseteq \mathcal{F}_{K}G$. Let $M$ be any countable $\mathbb{Z}G$-module, we denote by $\widehat{M}$ its Pontryagin dual and $G\curvearrowright \widehat{M}$ the associated algebraic action (see \cite[Section 13.1]{kl}). Then the algebraic action $G\curvearrowright \widehat{\mathbb{Z}G\otimes_{\mathbb{Z}K}P}$ is conjugate to the coinduced action $G\curvearrowright (\widehat{P})^{G/K}$ by \cite[Proposition 6.23]{ben}. Now, we define a continuous cocycle $c': G\times (\widehat{P})^{G/K}\to \{\pm 1\}\subseteq \mathbb{T}$ by setting $c'(g, y)=y(c(g^{-1}))$ for all $g\in G$, $y\in Y=(\widehat{P})^{G/K}$.

Then the proof goes the same as the proof of Theorem 2 in \cite{cj}, we include the sketch below.

First, we claim that $c'$ is not trivial as a $\{\pm 1\}\cong \mathbb{Z}_2$-valued cocycle.

Assume not, then there exists a group homomorphism $\rho: G\to \mathbb{Z}_2$ and a continuous map $b: Y\to \mathbb{Z}_2\subset \mathbb{T}$ such that $c'(g, y)=b(gy)^{-1}\rho(g)b(y)$ for all $g\in G$ and all $y\in Y$. Now we apply the above equation to $y'=(0)_{G/K}$, where $0$ is a fixed point of the algebraic action $K\curvearrowright \widehat{P}$. Since $y'$ is $G$-invariant, we deduce $\rho(g)=1$ for all $g\in G$. Hence, $y(c(g^{-1}))=b(gy)^{-1}b(y)$ for all $g\in G$ and all $y\in Y$.

Next, we check we may assume $b\in \mathbb{Z}G\otimes_{\mathbb{Z}K}P$.

To see this, use the fact that $y_1(c(g^{-1}))y_2(c(g^{-1}))=(y_1+y_2)(c(g^{-1}))$ for all $y_1$, $y_2\in Y$. Hence, $b(gy_1)^{-1}b(y_1)b(gy_2)^{-1}b(y_2)=b(gy_1+gy_2)^{-1}b(y_1+y_2)$. Equivalently, $b(y_1)b(y_2)b(y_1+y_2)^{-1}=b(gy_1)b(gy_2)b(gy_1+gy_2)^{-1}$ for all $g\in G$ and $y_1$, $y_2\in Y$.

Now, since $\widetilde{e}(G, K)\neq 0$ implies $[G: K]=\infty$, Proposition \ref{product measure for coinduced actions} implies that the coinduced action is weakly mixing with respect to the $G$-invariant product measure $\mu:=\nu^{G/K}$, where $\nu$ is the Haar measure on $\widehat{P}$. Therefore, $b(y_1)b(y_2)b(y_1+y_2)^{-1}=\lambda$ for some constant $\lambda\in \mathbb{T}$ and $\mu\times \mu$-a.e. $(y_1, y_2)\in Y\times Y$. Since $b$ is continuous, after replacing $b$ with $\lambda^{-1}b$, we may assume $b\in \widehat{Y}=\mathbb{Z}G\otimes_{\mathbb{Z}K}P$. 

By Pontryagin duality, we deduce $c(g)=b-gb$ for all $g\in G$. A contradiction.\qedhere
\end{proof}

\section{Remarks}\label{section on remarks}
We conclude this note with some remarks.

(1) Coinduced actions may not be conjugate to full shifts. To see this, we can take $K\curvearrowright X$ to be a compact algebraic action of any countably infinite amenable group $K$. This action has zero entropy by \cite[Theorem 9.11]{kl}, then apply \cite[Proposition 6.22]{ben} to deduce that a coinduced action of this action to any larger sofic group still has zero entropy. 

(2) It is not clear whether the assumption of being ``commensurated" or the action $K\curvearrowright X$ has a fixed point can be removed. One may wonder whether one can use directly the Schreier coset graph associated to any subgroup $K$ rather than the one in Section \ref{section on geometry of coset graph}, the trouble is that it is not clear how to prove Lemma \ref{lem 3.3} without assuming any suitable normality condition on $K$.

(3) We list some examples of commensurated subgroups with relative one end below.

\begin{itemize}
\item Let $K$ be any finite subgroup and $G$ be a finitely generated group with one end. Then $\widetilde{e}(G, K)=e(G)=1$ by Proposition \ref{property of relative ends} (7).
\item Let $K$ be a finitely generated normal subgroup of $G$ such that $G/K$ has one end. Then $\widetilde{e}(G, K)=e(G/K)=1$ by Proposition \ref{property of relative ends} (8).
\end{itemize}
It is interesting to find more group pairs $(K, G)$ satisfying the above assumptions. One may also expect to give a combinatorial criterion for having $\widetilde{e}(G, K)=1$ similar to \cite[Theorem 9]{aas} for $e(G)=1$. Following the proof of \cite[Theorem 9]{aas}, one need to find a relative version of Stallings's theorem to characterize the pair $K\leq G$ such that $\widetilde{e}(G, K)\geq 2$, i.e. one expects $G$ splits over some subgroup related to $K$. This is loosely related to Kropholler's conjecture, see \cite{dun, kr} and reference therein for more discussion on this conjecture, I have not studied this question further. But one can show some group pairs $(K, G)$ appeared in \cite{d} satisfy $\widetilde{e}(G, K)=1$ indirectly.\\

\textbf{Acknowledgements:} Part of this paper was finished while the author was in SUNY at Buffalo, he thanks his advisor Prof. Hanfeng Li for support there. The author was supported by Science Research Center Program through NRF funded by the Ministry of Science, ICT \& Future Planning (No. NRF-2016R1A5A1008055). He also thanks the anonymous referee for helpful comments.

\end{document}